\documentclass[12pt]{article}
\usepackage{amssymb,latexsym,amsfonts,amsmath,amsthm,verbatim,eucal,hyperref}

\def\PP{\mathbb{P}}
\def\EE{\mathbb{E}}

\def\RR{\mathbb{R}}
\def\NN{\mathbb{N}}

\def\1{\mathbf{1}}
\def\eps{\varepsilon}

\def\<{\langle}
\def\>{\rangle}

\def\mes{\operatorname{mes}}

\theoremstyle{plain}
\newtheorem{thm}{Theorem}
\newtheorem*{thm*}{Theorem}
\newtheorem{lemma}[thm]{Lemma}
\newtheorem*{lemma*}{Lemma*}

\newtheorem*{prop*}{Proposition}

\newtheorem*{cor*}{Corollary}

\newtheorem*{cl*}{Claim}
\theoremstyle{remark}

\newtheorem*{not*}{Notation}
\theoremstyle{definition}

\newtheorem*{ex*}{Example}

\newtheorem*{qn*}{Question}

\begin{document}

\title{Poisson asymptotics for random projections \\ of points on a high-dimensional sphere}
\author{Itai Benjamini$^1$, Oded Schramm$^2$, and Sasha Sodin$^3$}
\maketitle

\begin{abstract}
Project a collection of points on the high-dimensional sphere onto
a random direction. If most of the points are sufficiently far from one
another in an appropriate sense, the projection is locally
close in distribution to the Poisson point process.
\end{abstract}

\footnotetext[1]{Weizmann Institute, itai.benjamini@weizmann.ac.il}
\footnotetext[2]{Oded died  while solo climbing Guye Peak in Washington State on 
September 1, 2008.}
\footnotetext[3]{Tel Aviv University, sodinale@post.tau.ac.il.
Supported in part by the Adams Fellowship Program of the Israel
Academy of Sciences and Humanities and by the ISF.}

\section{Introduction}

Let $x_1, x_2, \cdots, x_n$ be $n$ points on the $(d-1)$-dimensional sphere $S^{d-1}$.
We assume that $n$, $d$, and the points themselves depend on an implicit parameter, so that
$d \to \infty$, $n \to \infty$. Consider the normalised projections $\< x_j, \sqrt{d} U \>$
of the points onto a random direction $U$ on the sphere.

\vspace{1mm}\noindent
Fix $a \in \RR$, and denote
\[ \xi_{a} = \xi_a(U) = \sum_{j=1}^n \delta \Big(\bullet - n(\< x_j, \sqrt{d}\, U\>-a) \Big)~. \]
This is a {\em point process}, i.e.\ a random locally finite integer-valued Borel measure on $\RR$.
The (homogeneous) Poisson process with intensity $\lambda$ ($\lambda > 0$) is a point process $\eta$
such that
\begin{equation}\label{eq:.5}
\eta(B) \sim \mathrm{Pois}(\lambda \cdot \mes B)
\end{equation}
for any Borel set $B \subset \RR$. The reader may find further properties of Poisson point processes
in the book of Reiss \cite{R}.

\begin{thm*} Assume that, for any $\eps > 0$,
\begin{equation}\label{eq:0''}
    {\sum}^{\eps} \min \left( \frac{1}{\sqrt{1 - |\< x_i, x_j \>|}}, \,\,  n \right) = o(n^2)~,
\end{equation}
where the sum is over all pairs $i < j$ such that $|\< x_i, x_j\>| \geq \eps$. Then, for any
$a \in \RR$, $\xi_a$ converges to the Poisson process with intensity
\[ \phi(a) = (2\pi)^{-1/2} \exp(-a^2/2)~, \]
in the following sense\footnote{that is stronger than weak convergence in distribution}:
for any bounded Borel set $I \Subset \RR$,
\[ \xi_a(I) \overset{D}{\longrightarrow} \mathrm{Pois}(\phi(a) \, \mes I)~.\]
\end{thm*}

That is, the random variables $\< x_j, \sqrt{d}\, U\>$ behave locally as independent samples
from the Gaussian distribution.

\begin{ex*} It is not hard to see that the condition (\ref{eq:0''}) is fulfilled
for the vertices of the discrete cube:
\begin{equation}\label{eq:cube}
n = 2^d~, \quad x_j = (\pm 1/\sqrt{d}, \cdots, \pm 1/\sqrt{d})~.
\end{equation}
\end{ex*}

\section{Proof of Theorem}\label{s:p}

The proof is based on the following elementary (and well-known) lemma, sometimes
referred to as Archimedes' theorem. The lemma follows from the fact (discovered by
Archimedes for $d = 3$) that the projection of the uniform measure on $S^{d-1}$ onto
a $(d-2)$-dimensional subspace is the uniform measure on the unit ball of this subspace.
\begin{lemma}\label{lemma}
Let $y_1, \cdots, y_k$ be pairwise distinct points on the sphere $S^{d-1}$, $k \leq d-2$.
The projections $H_j = \< y_j, \sqrt{d}U \>$ of $y_j$ on a random direction have
joint density
\[ p(h) = \frac{\Gamma(d/2)}{\Gamma((d-k)/2)} \, \frac{1}{(\pi d)^{k/2}} \,\,
    {\det}^{-1/2} M \,\, (1 - |M^{-1/2}h|^2/d)_+^{\frac{d-k-2}{2}}~,\]
where $M_{j j'} = \< y_j, y_{j'} \>$.
\end{lemma}

\noindent Let $I \Subset \RR$, and fix $k \in \NN$. Denote by $N_I$ the number
of points of $\xi_a$ in $I$, $N_I = \xi_a(I)$. Let us show that $N_I$ converges
in distribution to the Poisson law.

\vspace{2mm} \noindent {\bf Step 1:} First, let us assume that
\begin{equation}\label{eq:str}
\max_{j \neq j'} |\<x_j, x_{j'}\>| = o(1)~.
\end{equation}
Then proceed as follows:
\begin{equation}\label{eq:1}
\begin{split}
\EE \binom{N_I}{k}
    &= \EE \sum_{1 \leq j_1 < \cdots < j_k \leq n}
        \prod_{s=1}^k \1_I (n(\< x_{j_s}, \sqrt{d}\, U\>-a)) \\
    &= \sum \EE \prod_s \1_{a + n^{-1}I}(\< x_{j_s}, \sqrt{d}\, U\>)~.
\end{split}
\end{equation}
Denote $H_s = \< x_{j_s}, \sqrt{d}\, U\>$, and let $M_{ss'} = \< x_{j_s}, x_{j_{s'}}\>$.
According to Lemma~\ref{lemma}, the joint density of $H = (H_1, \cdots, H_k)$ is equal to
\[ p(h) = \frac{\Gamma(d/2)}{\Gamma((d-k)/2)} \, \frac{1}{(\pi d)^{k/2}} \,
    {\det}^{-1/2} M \, (1 - |M^{-1/2}h|^2/d)_+^{\frac{d-k-2}{2}}~.\]

Now, according to (\ref{eq:str}), $M = \1 + o(1)$, where the $o(1)$ term tends to zero
entry-wise and hence also in norm (recall that $k$ is fixed). Thus
\[ p(h) = \frac{\Gamma(d/2)}{\Gamma((d-k)/2)} \, \frac{1}{(\pi d)^{k/2}} \,
    (1 - (1+o(1))|h|^2/d)_+^{\frac{d-k-2}{2}} \, (1+o(1))~,\]
where the $o(1)$ term is uniform in $h$. Recalling that $d \to \infty$ (whereas $k$ is fixed),
we see that
\[\begin{split}
p(h) &= \frac{\Gamma(d/2)}{\Gamma((d-k)/2)} \, \frac{1}{(\pi d)^{k/2}} \,
    \exp(- |h|^2/2) \,\, (1 + o(1))  \\
     &= (2\pi)^{-k/2} \, \exp(- |h|^2/2) \,\, (1 + o(1))~,
\end{split}\]
uniformly on compact subsets of $\RR^k$. Therefore
\[ \EE \prod_{s=1}^k \1_{a + n^{-1}I}(H_s) = (\gamma(a + n^{-1}I))^k \, (1+o(1))~, \]
where $\gamma = N(0, 1)$ is the standard Gaussian measure. The set $I$ is bounded and fixed,
whereas $n \to \infty$, hence
\[ \EE \prod_{s=1}^k \1_{a + n^{-1}I}(H_s)
    = \left( \phi(a) \frac{\mes I}{n} \right)^k (1 + o(1))~.\]
Returning to (\ref{eq:1}), we deduce:
\begin{equation}\label{eq:2}
\EE \binom{N_I}{k} = \binom{n}{k} \left( \phi(a) \frac{\mes I}{n} \right)^k (1 + o(1))
    = \frac{(\phi(a)\mes I)^k}{k!} (1+o(1))~.
\end{equation}
That is, the factorial moments of $N_I$ tend to those of the Poisson distribution
$\textrm{Pois}(\phi(a)\mes I)$. The Poisson distribution has (better than) exponential tails, thus
$N_I$ converges in distribution to $\textrm{Pois}(\phi(a) \mes I)$.

\vspace{2mm}\noindent
{\bf Step 2}: Now let us relax the assumption (\ref{eq:str}). First, (\ref{eq:0''})
implies that one can choose $\eps \to 0$ so that
\begin{equation}\label{eq:0''+}
    {\sum}^\eps \min \left( \frac{1}{\sqrt{1 - |\< x_i, x_j \>|}}, \,\,  n \right) = o(n^2)~.
\end{equation}
Let
\[ p_I(x, x') = \PP \left\{ \<x, \sqrt{d} U\>, \<x', \sqrt{d} U\> \in a + n^{-1} I\right\}~. \]

\begin{lemma}
$p_I(x, x') \leq C_I \min \left( n^{-2} (1 - |\< x, x' \>|)^{-1/2}, n^{-1} \right)$.
\end{lemma}
\begin{proof}
By Lemma~\ref{lemma}, the joint density of $\<x, \sqrt{d} U\>, \<x', \sqrt{d} U\>$ is given by
\[ p(h) = \frac{\Gamma(d/2)}{\Gamma((d-2)/2)} \, \frac{1}{\pi d} \, {\det}^{-1/2} M \,
    \left( 1 - |M^{-1/2} h|^2 / d \right)^{\frac{d-4}{2}}~,\]
where
\[ M = \left(
    \begin{array}{cc} 1 &\< x, x' \> \\ \<x, x'\> &1 \end{array} \right)~.\]
Therefore
\[ {\det}^{-1/2} M = (1 - \< x, x'\>^2)^{-1/2} = O(1) \cdot (1 - |\< x, x'\>|)^{-1/2}~,\]
and $p(h) = O(1) \, (1 - |\< x, x'\>|)^{-1/2}$. Thus
\[ p_I(x, x') = \iint_{(a + n^{-1}I)^2} p(h) dh
    \leq C_I n^{-2} (1 - |\< x, x'\>|)^{-1/2}~. \]
Also,
\[ p_I(x, x') \leq \PP \left\{ x \in a + n^{-1}I \right\} \leq C_I n^{-1}~.\]
\end{proof}

\noindent According to the lemma and (\ref{eq:0''+}), $\PP(A) = o(1)$, where
\[ A = \left\{ \exists j \neq j' \,\, \big| \,\, |\< x_j, x_{j'} \>| \geq \eps, \,\,
    \<x_j, \sqrt{d} U\>, \<x_{j'}, \sqrt{d}U\> \in a + n^{-1}I\right\}~. \]
Repeating the argument of Step~1, we see that the conditional distribution of $N_I$ given
$\neg A$ (the negation of $A$) tends to $\textrm{Pois}(\phi(a) \, \mes I)$. Thus the same is true for
$N_I$ itself.

\qed

\section{Some remarks}

\newcounter{remnr}
\def\rem{
    \addtocounter{remnr}{1}
    \vspace{2mm}\noindent{\bf \arabic{remnr}.} }

\rem Diaconis and Freedman \cite{DF} have proved the following: if, for any $\eps > 0$,
\begin{eqnarray}
\label{eq:df.1}
    &&\# \left\{ j \, | \, \big||x_j|^2 - 1\big| > \eps \right\} = o(n)~, \\
\label{eq:df.2}
    &&\# \left\{ j,k \, | \, \big|\<x_j, x_k\>\big| > \eps \right\} = o(n^2)~,
\end{eqnarray}
then the empirical distribution
\[ n^{-1} \, \sum_j \delta \big( \bullet - \< x_j, \sqrt{d}U \> \big)\]
converges (weakly, in distribution) to the standard Gaussian law. Our result
can be seen as a local version of this statement.

\rem The conditions (\ref{eq:df.1})-(\ref{eq:df.2}) are not sufficient for the conclusion
of our theorem, as one can see from the following example:
\[ (x_1, x_2, \cdots, x_n) = (e_1, \cdots, e_d, e_1, \cdots, e_{\lfloor \delta d \rfloor})\]
(where $(e_1, \cdots, e_d)$ is the standard basis in $\RR^d$, and $0 < \delta < 1$ is an arbitrary
constant.)

\rem The assumption that $x_j \in S^{d-1}$ in our theorem can be relaxed. For example, the
Diaconis--Freedman assumption (\ref{eq:df.1}) is sufficient for $a \neq 0$.

\rem For any $\delta > 0$, one can construct a $\delta$-net on $S^{d-1}$ for which the
assumption (\ref{eq:0''}) is satisfied. Indeed, if the distribution of the points in the net
is sufficiently regular,
\[\begin{split} {\sum}^\eps \frac{1}{\sqrt{1- |\<x_i, x_j \>|}}
    &\approx n^2 \displaystyle\iint_{|\<x, y\>| \geq \eps}
        \frac{d\sigma_d(x) d\sigma_d(y)}{\sqrt{1 - |\<x, y\>|}} = o(n^2)
\end{split}\]
as $d\to\infty$.

\rem The theorem and the proof can be easily extended to random projections onto
an $r$ dimensional subspace, where $r$ is any number (fixed, or slowly growing with $d$).

\rem One may ask whether it is possible to reduce the randomness in the conclusion of
the theorem, and still have (at least, weak) convergence to the Poisson process.
For example, one may project the point $x_j$ onto a random Bernoulli direction
$B = (\pm 1, \cdots, \pm 1)$. Even for the points (\ref{eq:cube}), the limit will not be
Poisson, since all the projections will be integer multiples of $1/\sqrt{d}$.
Instead, one can consider a random perturbed Bernoulli direction:
$B_\eps = (\pm (1+\eps_1), \cdots, \pm (1+\eps_d))$. Is it true that the projections
of the points (\ref{eq:cube}) are asymptotically Poisson for a `generic' perturbation
$\eps$? Is there a natural arithmetic condition on $\eps$ that ensures that the projections
are asymptotically Poisson?

\rem It may also be interesting to consider projections of points $\{x_j\}$ for which
the condition (\ref{eq:0''}) is violated. Which point processes can appear in the
limit, as $n,d \to \infty$?

\vspace{2mm} \noindent
{\bf Acknowledgment.} We are grateful to Ron Peled and Ofer Zeitouni for their helpful
remarks.

\end{document}